\documentclass[12pt,leqno]{article}

\usepackage{amsmath,amssymb,amsthm}
\usepackage[all]{xy}

\makeatletter

\newtheorem*{theorem*}{Theorem}
\newtheorem*{main-theorem}{Theorem~\ref{main.thm}}

\newtheorem{theorem}{Theorem}[section]
\newtheorem{proposition}[theorem]{Proposition}
\newtheorem{lemma}[theorem]{Lemma}
\newtheorem{corollary}[theorem]{Corollary}

\theoremstyle{definition}

\newtheorem{example}[theorem]{Example}
\newtheorem{definition}[theorem]{Definition}

\theoremstyle{remark}
\newtheorem{remark}[theorem]{Remark}

\def\({{\rm (}}
\def\){{\rm )}}

\let\Mathrm\operator@font
\let\Cal\mathcal
\let\Bbb\mathbb
\newcommand{\fm}{\ensuremath{\mathfrak m}}

\def\standop#1{\mathop{\Mathrm #1}\nolimits}
\def\difstop#1#2{\expandafter\def\csname #1\endcsname{\standop{#2}}}
\def\defstop#1{\difstop{#1}{#1}}

\defstop{add}
\def\abs#1{\lvert{#1}\rvert}

\defstop{Coker}

\defstop{FL}
\defstop{FS}

\defstop{End}

\defstop{Gr}

\defstop{Hom}
\newcommand{\HK}{_\mathrm{HK}}

\difstop{Image}{Im}
\difstop{indecomp}{Ind}
\defstop{Iso}
\defstop{Ker}

\difstop{md}{mod}
\defstop{Mod}

\defstop{NF}
\def\norm#1{\lVert{#1}\rVert}

\defstop{rank}
\defstop{Ref}
\def\RR{{\mathbb R}}

\defstop{Spec}
\defstop{Sym}
\difstop{sm}{sum}

\defstop{Tor}

\defstop{vol}

\def\ZZ{{\mathbb Z}}

\def\C{\Cal C}
\def\D{\Cal D}
\def\F{\Cal F}
\def\M{\Cal M}
\def\P{\Cal P}
\def\V{\Cal V}
\def\fm{\mathfrak{m}}

\def\fq{\mathfrak{q}}
\def\fd{\mathfrak{d}}

\def\section{\@startsection{section}{1}{\z@ }%
  {-3.5ex plus -1ex minus -.2ex}{2.3ex plus .2ex}{\bf }}

\long\def\refname{\par\kern -3ex
  \begin{center}\rm R\sc{eferences}\end{center}\par\kern 
  -2ex}

\def\@seccntformat#1{\csname the#1\endcsname.\quad}

\def\@@@sect#1#2#3#4#5#6[#7]#8{%
  \ifnum #2>\c@secnumdepth 
  \def \@svsec {}\else \refstepcounter {#1}%
  \def\@svsec{}
  \fi 
  \@tempskipa #5\relax 
  \ifdim \@tempskipa >\z@ 
  \begingroup #6\relax \@hangfrom {\hskip #3\relax 
    \@svsec}{\interlinepenalty \@M #8\par }\endgroup 
  \csname #1mark\endcsname {#7}
  \else 
  \def \@svsechd {#6\hskip #3\@svsec #8\csname #1mark\endcsname {#7}}
  \fi \@xsect {#5}}

\def\@@@startsection#1#2#3#4#5#6{%
  \if@noskipsec \leavevmode \fi \par \@tempskipa #4\relax \@afterindenttrue 
  \ifdim \@tempskipa <\z@ \@tempskipa -\@tempskipa \@afterindentfalse 
  \fi \if@nobreak \everypar {}\else \addpenalty {\@secpenalty }\addvspace 
  {\@tempskipa }\fi \@ifstar {\@ssect {#3}{#4}{#5}{#6}}{\@dblarg 
    {\@@@sect {#1}{#2}{#3}{#4}{#5}{#6}}}}

\def\theparagraph{\thesection.\arabic{paragraph}}
\def\aparagraph{\@@@startsection{paragraph}{2}{\z@ }%
  {1.75ex plus .2ex minus .15ex}{-1em}{\bf(\theparagraph) } }
\def\paragraph{\@@@startsection{paragraph}{2}{\z@ }%
  {1.75ex plus .2ex minus .15ex}{-1em}{}{\bf(\theparagraph)} }

\let\c@theorem\c@paragraph

\title{The Asymptotic Behavior of Frobenius Direct Images of Rings of Invariants}
\author{%
M{\sc itsuyasu} H{\sc ashimoto}
\and
P{\sc eter} S{\sc ymonds}
}

\date{}
  
\begin{document}

\maketitle

\begin{abstract}
  We define the Frobenius limit of a module over a ring of prime characteristic
  to be the limit of the
  normalized Frobenius direct images in a certain Grothendieck group. When a finite group acts on a polynomial ring, we calculate this limit for all the modules over the twisted group algebra that are free over the polynomial ring; we also calculate the Frobenius limit for the restriction of these to the ring of invariants. 
  As an application, we generalize the description of the generalized $F$-signature of a ring of invariants
  by the second author and Nakajima to the modular case.
\end{abstract}

\section{Introduction}
\paragraph
In commutative algebra, the study of the
asymptotic behavior of the Frobenius direct images of
a ring of prime characteristic $p$ (or a module over it) has been very fruitful.
This includes the study of
invariants such as the Hilbert--Kunz multiplicity \cite{Monsky} and
the $F$-signature \cite{HL} and its variants \cite{Sannai, HN}.

These invariants have been studied
for the ring of invariants of a finite group acting on a ring, see
\cite[(2.7), (5.4)]{WY}, \cite[Example~18]{HL}, \cite[(4.2)]{WY2}, \cite[(3.9)]{HN},
and \cite{Nakajima}.

\paragraph\label{TS.par}
Let $T=\bigoplus_{n\geq 0}T_n$ be a graded Noetherian commutative ring,
where $T_0$ is a finite direct product of Henselian local rings.
Let $S=\bigoplus_{n\geq 0}S_n$
be a finite graded $T$-algebra, which might not be commutative.

Let $\Theta^*(S)$ denote the Grothendieck group of 
the commutative monoid of finitely generated $\mathbb Q$-graded $S$-modules under direct sum,
but tensored
with $\RR$;
this means that $\Theta^*(S)$ is the $\RR$-space generated by
the finitely generated $\mathbb Q$-graded $S$-modules
subject to the relations $[M]=[M_1]+[M_2]$ whenever $M \cong M_1 \oplus M_2$.
We define $\Theta^\circ(S)$ to be the quotient of $\Theta^*(S)$ by
the relation $[M]=[M[\lambda]]$ for a finitely generated $\Bbb Q$-graded $S$-module $M$
and $\lambda\in\Bbb Q$, where $?[\lambda]$ denotes shift of degree by $\lambda$.

Because of our hypotheses on $S$,
the Krull--Schmidt property holds and so the finitely generated
indecomposable $\Bbb Q$-graded modules form a basis for $\Theta^*(S)$.
Thus $\indecomp^\circ(S)$,
the set of indecomposable $\Bbb Q$-graded modules modulo shift of degree, forms a basis for
$\Theta^\circ(S)$.
For $\alpha\in \Theta^\circ(S)$ we can write
\[
\alpha=\sum_{M\in \indecomp^\circ S}c_M[M] \qquad (c_M\in\RR)
\]
uniquely.
We define $\norm\alpha _S:=\sum_M\abs{c_M}u_S(M)$,
where $u_S(M)$ denotes $\ell_S(M/\fm_S M)$,
where $\fm_S=S_++J(S_0)$ is the graded Jacobson radical of $S$ and $\ell_S$ denotes
the length function.
It is easy to see that $(\Theta(S),\norm{\;\cdot\;})_S$ is a normed space.

\paragraph
Now let $k$ be an $F$-finite (that is, $[k:k^p]<\infty$) field of characteristic $p$,
and $R=\bigoplus_{n\geq 0}R_n$ a graded Noetherian commutative ring such that
$R_0$ is an $F$-finite Henselian local ring.
We assume that $\dim_k R_0/J(R_0)<\infty$.
Let $G$ be a finite group acting on $R$ as $k$-algebra automorphisms.
Let $S=R*G$ and $T=R^G$.
Then $T$ and $S$ are as in (\ref{TS.par}).

Let $d$ be the Krull dimension of $R$.
Set $\fd:=\log_p[k:k^p]$, and $\delta:=d+\fd$.
For any finitely generated $S$-module $M$,
we define the {\em Frobenius limit} of $M$ to be
\[
\FL(M)=\lim_{e\rightarrow \infty}\frac{1}{p^{\delta e}}[{}^eM]
\]
in $\Theta(S)$, provided that this limit exists,
where ${}^eM$ is the $e$th Frobenius direct image of $M$.
Note that $\FL(M)$ is considered to be the limit of the modules themselves,
rather than of some numerical invariant.
If the ring is commutative and $\FL(M)$ exists, the Hilbert--Kunz multiplicity and 
the (generalized) $F$-signature can be read off from it; see section~\ref{quotients.sec}.

\paragraph
Suppose that $R$ be commutative.
The group $\Theta(R)$ is larger than the Grothendieck group
$G_0(R)_\RR$, where the relations come from short exact sequences.
The latter is isomorphic to $A_*(R)_\RR$, the Chow group of $R$ (tensored with
$\RR$) through the Riemann--Roch isomorphism $\tau_R$, see \cite{Fulton}.
Let us write $\tau_R([R])=c_d+c_{d-1}+\cdots+c_0$, where $c_i$ is the
component of dimension $i$.
Then $\tau_R(\FL[R])$ is just $c_d$, which plays an important role in the
intersection theory of commutative algebra, see \cite[(2.2)]{Kurano}
and \cite{KO}.

Bruns gave a formula for $\FL(R)$ for a normal affine
semigroup ring (although
he did not define Frobenius limits, he proved a 
theorem \cite[Theorem~3.1]{Bruns} giving some more information
than $\FL(R)$, see Example~\ref{Bruns.ex}).

\paragraph
Now suppose that a finite group $G$ acts faithfully on a graded polynomial ring $B$,
so we can form the twisted group algebra $B * G$.
The generators of $B$ must be in positive degrees, but not necessarily all the same. Let $A=B^G$.

\begin{theorem*}[(\ref{thm.FLBG}), (\ref{thm.FLA})]
  Suppose that $F$ is a $\mathbb Q$-graded $B * G$-module that is free of rank $f$ over $B$.
  Then the $F$-limits of $[F]$ and $[F^G]$ exist and
\[
\FL(F)=\frac{f}{|G|} [B * G]
\]
in $\Theta^\circ (B * G)$ and 
\[
\FL(F^G)=\frac{f}{|G|} [B]
\]
in $\Theta^\circ (A)$. Analogous formulas hold after completion at the irrelevant ideal.
\end{theorem*}

As a consequence we obtain the following theorem.

\begin{theorem*}[(\ref{main.thm})]
  Let $k=V_0,V_1,\ldots,V_n$ be the simple $kG$-modules.
  For each $i$, let
  $P_i\rightarrow V_i$ be the projective cover, and
  $M_i:=(B\otimes_k P_i)^G$.
  Suppose that $F$ is a $\mathbb Q$-graded $B * G$-module that is free of rank $f$ over $B$.
  Then the $F$-limit of $[F^G]$ exists, and
  \[
  \FL([ F^G])=\frac{f}{\abs{G}}[ B]
  =\frac{f}{\abs{G}}\sum_{i=0}^n \frac{\dim_k V_i}{\dim_k \End_{kG}(V_i)}[\hat M_i]
  \]
 in $\Theta^\circ(A)$. The analogous formula holds after completion at the irrelevant ideal. 
\end{theorem*}

In particular, we have a formula for $\FL[A]$ and $\FL([\hat A])$: see Corollary~\ref{main.cor}.

Using this theorem, we generalize a result on the generalized $F$-signature
\cite[(3.9)]{HN} to the modular case (Corollary~\ref{HN-generalization.cor}).
We also get a new proof of the theorem of Broer \cite{Broer} and Yasuda \cite{Yasuda}
which says that if $G$ does not have a pseudo-reflection and $p$ divides
the order $\abs{G}$ of $G$, then $A$ is not weakly $F$-regular.

For another application of this work to invariant theory,
see \cite{Hashimotoxx}.

In section~\ref{Frobenius.sec}, we fix our notation for Frobenius direct images.
In section~\ref{quotients.sec}, we study the group $\Theta(S)$ and define the Frobenius
limits.
In section~\ref{main.sec}, we prove the main theorems and in section~\ref{applications.sec}
we derive some consequences.

Acknowledgments:
the authors are grateful to
Professor Kazuhiko Kurano for his valuable advice.

\section{Rings, modules and Frobenius direct image}\label{Frobenius.sec}

\paragraph
\label{modules.par}
Let $k$ be a field.
By a module over a ring we mean a left module, unless otherwise specified.
A graded ring means a ring graded by the semigroup of non-negative integers.
Modules will be graded by $\mathbb Q$;
since we only consider finitely generated modules, the graded pieces are only
non-zero on a discrete subgroup, which is contained in $\frac{1}{r} \mathbb Z$
for some $r \in \mathbb N$.
The morphisms are degree preserving.
Let $G$ be a finite group acting on a ring $R$.
By an $(R,G)$-module $M$, we mean an $R$-module that is also a $kG$-module in such a way that
$g(rm)=(gr)(gm)$, $g \in G, r \in R, m \in M$.
If $M$ is an $(R,G)$-module and $V$ a $G$-module, then $M\otimes_k V$ is an $(R,G)$-module by
$r(m\otimes v)=rm\otimes v$ and $g(m\otimes v)=gm\otimes gv$ for
$r\in R$, $m\in M$, $v\in V$, and $g\in G$.

\paragraph
By a virtually commutative ring we mean a ring $S$
that contains some central subalgebra
$T$
such that $S$ is finite over $T$.
The example we have in mind is when $G$ acts on a commutative ring
$R$ and $S$ is the twisted group algebra $R*G$.
That is, $R*G=\bigoplus_{g\in G}Rg$ as an $R$-module, and the product is
given by $(rg)(r'g')=(r(gr'))(gg')$.
The ring $R*G$ is finite over the ring of invariants $T=R^G$
in many cases.
For example, assume that $R$ is a commutative Noetherian
$k$-algebra and the action of $G$ is by $k$-algebra automorphisms.
If $R$ is of finite type over $k$;
$R$ is complete with residue field $k$;
the characteristic of $k$ is $p>0$ and $R$ is $F$-finite (see \ref{F-finite.par})
\cite{Fogarty}, \cite[(9.6)]{Hashimoto12};
or the order of $G$ is not
divisible by the characteristic of $k$, then $R$ and $S=R*G$ are finite over
$T=R^G$.

An $R*G$-module is an $(R,G)$-module in an obvious way, and vice versa.
We identify these two objects.

\paragraph
Note that the $(G,R)$-module $R\otimes_k kG$ as an $R*G$-module is
identified with the rank-one free module $R*G$ by the obvious
map $r\otimes g\mapsto rg$.

\paragraph
Let $k$ be of characteristic $p>0$.
For a commutative $k$-algebra
$R$, the Frobenius homomorphism $F:R \rightarrow R$ is defined by $F(a)=a^p$.
For $r\in\Bbb Z$, let ${}^rR$ be a copy of the ring $R$, except that, in the graded case, the values of the grading are divided by $p^r$ (here we briefly suspend our convention that all rings are integer graded).
For any $e\geq 0$, we regard ${}^{r+e}R$ an ${}^r R$-algebra
through the Frobenius map $F^e:{}^rR=R\rightarrow R={}^{r+e}R$.
An $R$-module $M$, viewed as an ${}^rR$-module is denoted by ${}^rM$;
$m\in M$ is denoted by ${}^rm$ when it is viewed as an element of ${}^rM$.
When $e \geq 0$, we can regard ${}^eM$ as an $R$-module by $a({}^em)=F^e(a){}^em={}^e(a^{p^e}m)$.
Then $F^e({}^ra)={}^{r+e}(a^{p^e})=({}^{r+e}a)^{p^e}$.
The $R$-module ${}^eM$ is sometimes written as $F^e_*M$, and is called the $e$th
Frobenius direct image (also called Frobenius pushforward) of $M$.
If $R$ is graded,
$M$ is $\Bbb Q$-graded, and $m$ is a homogeneous element of degree $\lambda$, then
letting ${}^rm$ of degree $\lambda/p^r$, we have that ${}^r M$ is a $\Bbb Q$-graded ${}^r R$-module.
If $e\geq 0$, ${}^eM$ is a $\Bbb Q$-graded $R$-module via $F^e:R={}^0R\rightarrow {}^eR$.

\paragraph
If $V$ is a $k$-vector space
then ${}^eV$ is considered to be a $k$-vector space through the map $F^e$ for $e\geq 0$:
more explicitly, ${}^e v+{}^e v'={}^e (v+v')$ and
$\alpha \cdot {}^e v={}^e(\alpha^{p^e}v)$ for $\alpha\in k$ and $v,v'\in V$.
When $k$ is perfect, ${}^rV$ has a meaning for $r\in\Bbb Z$, and
it has the same dimension as $V$.
Note that ${}^eA$ is again a $k$-algebra, and
$F^{e}:{}^{e'}A\rightarrow {}^{e'+e}A$ is a $k$-algebra map for $e,e'\geq 0$.

\paragraph
In the notation above, ${}^0R$, ${}^0M$, ${}^0m$, and so on, are sometimes written as
$R$, $M$, $m$, and so on.

\paragraph
\label{frob.para}
Slightly more generally, for a commutative $k$-algebra $R$ and a finite group $G$
acting on $R$,
we define the Frobenius map $F=F_S$ of $S=R*G$ by
$F_S(\sum_{g\in G} r_g g) = \sum_g r_g^p g$.
If $G$ is trivial, then $R=S$, and $F_S$ is the usual Frobenius map.
Thus for an $R*G$-module $M$, ${}^eM$ is again an $R*G$-module.

\paragraph
Applying this to the group ring $kG$ (the case that $R=k$), we find that
${}^eV$ is a $kG$-module by $g\cdot {}^ev={}^e(gv)$ for $g\in G$ and $v\in V$.

If $V$ is $n$-dimensional, let $v_1,\ldots,v_n$ be a basis of $V$; then
we can write $gv_j=\sum_i c_{ij}v_i$.
If $k$ is perfect, then $g\cdot
{}^e v_j={}^e (gv_j)={}^e (\sum_i c_{ij}v_i)=\sum_i c_{ij}^{p^{-e}}{}^e v_i$.
Namely, ${}^eV$, as a matrix representation,
is obtained by taking the $p^e$th root of each matrix entry.

\begin{lemma}\label{F-twist.lem}
  Let $k$ and $G$ be as above.
  \begin{enumerate}
    \item[\bf 1] Let $V$ be a finite dimensional $G$-module.
  If $V$ is defined over $\Bbb F_q$, the field with $q=p^e$ elements,
  and $\fd:=\log_p [k:k^p]<\infty$,
  then ${}^eV\cong V^{p^{\fd e}}$.
\item[\bf 2] ${}^e(kG)\cong (kG)^{p^{\fd e}}$ for any $e\geq 0$.
  \end{enumerate}
\end{lemma}

\begin{proof}
  {\bf 1}.
  We set $r:=[{}^ek:k]=p^{\fd e}$.
  Let $V_0$ be the finite dimensional $\Bbb F_q$-module such that
  $k\otimes_{\Bbb F_q}V_0\cong V$.
  Then
  \[
    {}^eV\cong {}^ek\otimes_{\Bbb F_q}{}^eV_0\cong k^r\otimes_{\Bbb F_q}V_0\cong V^r.
  \]

  {\bf 2}.
  Since $kG$ is defined over $\Bbb F_p$, the assertion follows from {\bf 1}.
\end{proof}
        
\paragraph\label{F-finite.par}
$S$ is said to be $F$-finite if ${}^1S$ is a finite $S$-module.
If so, then $F^e:{}^rS\rightarrow {}^{r+e}S$ is finite for any $r\in\Bbb Z$ and $e\geq 0$.

\section{The Grothendieck group $\Theta(S)$}\label{quotients.sec}

\paragraph
Let $\C$ be an additive category.
We define its (additive) Grothendieck group to be
\[
[\C]:=(\bigoplus_{M\in \Iso \C}\ZZ\cdot M)
/(M-M_1-M_2\mid M\cong M_1\oplus M_2),
\]
where $\Iso \C$ is the set of isomorphism classes of objects in $\C$.
The class of $M$ in the group $[\C]$ is denoted by $[M]$.
We define $[\C]_\RR:=\RR\otimes_\ZZ[\C]$.
Note that we only have relations for split exact sequences,
not all exact sequences, even if $\C$ is abelian.

\paragraph
The group $[\C]$ is universal for additive maps from $\C$ to abelian groups, i.e.\
given an abelian group $\Gamma$ and
an additive map $f:\C\rightarrow \Gamma$ (that is, $f$ is a map $\C\rightarrow \Gamma$
such that $f(M)=f(M_1)+f(M_2)$ for every $M$, $M_1$, $M_2$ such that
$M\cong M_1\oplus M_2$),
$f$ extends to a unique homomorphism of abelian groups $f_*:[\C]\rightarrow\Gamma$. Thus $[\C ]_\RR$ is universal for additive maps to $\RR$-spaces. It follows that 
an additive functor $h:\C\rightarrow\D$ yields a homomorphism $h_*:[\C]\rightarrow [\D]$ which
maps $[M]$ to $[hM]$.

\begin{example}
  Let $S$ be a $k$-algebra.
  Let $S\md$ denote the category of finitely generated $S$-modules.
  Let $J(S)$ denote the Jacobson radical of $S$ and assume that $S/J(S)$ is finite
  dimensional over $k$. Then $u_{k,S}(M):= \dim_k(M/J(S)M)$ defines an additive
  function on $S\md$,
  which extends to $[S\md]_\RR$.

If $S$ is a commutative
integral domain
and we let $Q(S)$ denote the field of fractions of $S$,
then $\rank _S(M)=\dim_{Q(S)} Q(S) \otimes_S M$ is also additive and
extends to $[S\md ]_\RR$.
\end{example}

\paragraph
An additive category $\C$ is said to have the Krull--Schmidt
property if the endomorphism ring of any object
is semiperfect.
If so, the endomorphism ring of an indecomposable object is local, and hence the
Krull--Schmidt theorem holds, see \cite[(5.1.3)]{Popescu}.
Thus $[\C]$ is a $\Bbb Z$-free module with $\indecomp\C$ as free basis,
where $\indecomp \C$ is the set of isomorphism classes of indecomposable objects of 
$\C$
and $\indecomp\C$ is an $\RR$-basis of $[\C]_\RR$.

\paragraph
Let $T=\bigoplus_{n\geq 0}T_n$ be a commutative non-negatively graded Noetherian ring
(which might not be
a $k$-algebra) such that $T_0$ is a finite direct product of Henselian local rings.
Let $S=\bigoplus_{n\geq 0}S_n$ be a graded $T$-algebra that is a finite $T$-module.
For any finite graded $S$-module $M$, $\End_{S \Gr\md} M=(\End_S M)_0$ is a finite $T_0$-algebra
and is semiperfect \cite[(3.8)]{fac}, where $S\Gr\md$ is the category of graded finite $S$-modules.
Thus the Krull--Schmidt theorem holds for the category $S\Gr\md$;
see \cite{Popescu}[(5.1.3)].
Let $\fm_S$ denote the graded Jacobson radical $S_++J(S_0)$, where $S_+=\bigoplus_{n>0}S_n$ is the
irrelevant ideal.
We denote by $\hat ?$ the $\fm_S$-adic completion, which agrees with the $\fm_T$-adic
completion, where $\fm_T$ is the graded Jacobson radical of $T$.

\paragraph
We write $\Theta^*(S):=[S \Gr\md ]_\RR$, where $S \Gr\md $ is the category of $S$-finite
$\Bbb Q$-graded modules.
It will be convenient to consider the quotient of this where we identify any two
indecomposable modules that differ only by a shift in degree, which we denote by
$\Theta^\circ (S)$ or $\Theta(S)$.
We write $\Theta^\wedge(S):=[S\md]_\RR$, where $S\md$ is the category of $S$-finite
ungraded modules.

\paragraph
There is a sequence of natural maps
$\Theta^*(S)\rightarrow \Theta^\circ (S)
\rightarrow \Theta^\wedge(S)\rightarrow \Theta^\wedge (\hat{S})$.

\paragraph
It is easy to see that
if $S$ is concentrated in degree zero, then $\Theta^\circ=\Theta^\wedge$,
and the theory of $\Theta^\wedge$
for ungraded $S$ is contained in that of $\Theta^\circ$.

\paragraph
From now on we will assume that all our rings are of the type just described.
If 
$f:S'\rightarrow S$
is a finite degree-preserving map, there is a natural restriction map
$f^*:\Theta(S) \rightarrow \Theta(S')$ and the inflation map
$f_*:\Theta(S')\rightarrow \Theta(S)$.

If $I$ is an ideal in $S$ and $q:S \rightarrow S/I$
is the quotient map then we sometimes write
$\alpha /I\alpha$ for $q_*(\alpha)$.

\paragraph
For $\alpha\in \Theta^\circ(S)$, we can write
\[
\alpha=\sum_{[M]\in \indecomp^\circ S}c_M[M]
\]
uniquely, where $\indecomp^\circ(S)$ denotes
$\indecomp(S\Gr\md)/\mathord{\sim}$, where $M \sim M'$ if $M\cong M'[\lambda]$
for some $\lambda\in\Bbb Q$ ($?[\lambda]$ denotes shift of degree).
We define $\norm{\alpha} _S:=\sum_M\abs{c_M} u_S(M)$, where $u_S(M)= \ell_S(M/\fm_S M)$.
Then $(\Theta(S),\norm{\;\cdot\;} _S)$ is a normed space.
Thus $\Theta(S)$ becomes a metric space with the distance function $d$ given by $d(\alpha,\beta):=\norm{\alpha-\beta} _S$.

\begin{lemma}\label{RS.lem}
  Let $S$ be as above.
  \begin{enumerate}
    \item[\bf 1]
      Let $J$ be any ideal of $S$ such that there exists some $n\geq 1$ such that
      $\fm_S^n\subset J\subset \fm_S$.
      Define a norm $\norm{\;\cdot\;}_S^J$ on $\Theta(S)$
      by
      $\norm{\alpha}_S^J=\sum_M\abs{c_M}\ell_S(M/JM)$,
  where $\ell_S(-)$ denotes the length of an $S$-module.
  Then $\norm{\;\cdot\;}_S^J$ is equivalent to $\norm{\;\cdot\;}_S$.
\item[\bf 2]
  Let $f:S'\rightarrow S$ be a degree-preserving ring homomorphism such that
  $\fm_{S'} S\supset \fm_S^n$ for some $n\geq 1$ and
  $\fm_{S'}^m S\subset \fm_S$ for some $m\geq 1$ \(e.g.\ $S$ is $S'$-finite\).
  Define $\norm{\;\cdot\;}_{S'}^S$ by $\norm{\alpha}_{S'}^S=\sum_M\abs{c_M}\ell_{S'}(M/\fm_{S'} M)$.
  Then $\norm{\;\cdot\;}_{S'}^S$ is equivalent to $\norm{\;\cdot\;}_S$.
\item[\bf 3]
  Let $k$ be a field, and assume that $S$ is a $k$-algebra and $\dim_k S/\fm_S <\infty$.
  Define $\norm{\alpha}_{k,S}=\sum_M\abs{c_M}\dim_k M/\fm_S M$. Then
  $\norm{\;\cdot\;}_{k,S}$ is equivalent to $\norm{\;\cdot\;}_S$.
  \end{enumerate}
\end{lemma}

\begin{proof}
  {\bf 1}.
  For $M\in S\Gr\md$ we have $\ell_S(M/JM)\geq \ell_S(M/\fm_S M)$ and $\norm{\alpha}_S^J\geq \norm{\alpha}_S$ follows easily.
  There is a surjective map of graded $S$-modules
  $F\rightarrow M$, with $F$ free of rank $\ell_S(M/\fm_S M)$, which induces 
  a surjection $F/\fm_S^nF\rightarrow M/\fm_S^nM$.
  Setting $r:=\ell_S(S/\fm_S^n)$, we obtain
  $\ell_S(M/JM)\leq \ell_S(M/\fm_S^nM)\leq \ell_S(F/\fm_S^nF)=r\ell_S(M/\fm_S M)$,
  and $\norm{\alpha}_S^J\leq r\norm{\alpha_S}$ follows easily.
  It follows that $\norm{\;\cdot\;}_S^J$ is equivalent to $\norm{\;\cdot\;}_S$, as required.

  {\bf 2}.
  Let $T'$ be the center of $S'$.

  First we assume that $S$ is $S'$-finite (or equivalently, $T'$-finite) and show that the hypothesis on $f$ is satisfied.
  If $\fm_{T'}S\not\subset\fm_S$, then there exists some $a\in\fm_{T'}$ such that
  the ideal $a(S/\fm_S)$ of $S/\fm_S$ is nonzero.
  As $S/\fm_S$ has finite length, $a^n(S/\fm_S)=a^{n+1}(S/\fm_S)$ for some $n\geq 1$,
  then by the graded Nakayama's lemma $a^n(S/\fm_S)=0$.
  Since $S/\fm_S$ is semisimple, $a(S/\fm_S)$ is an idempotent ideal and so $a^n(S/\fm_S)\neq 0$, a contradiction.
  Therefore $\fm_{T'}S\subset \fm_S$.
  Note that $S/\fm_{T'} S$ is a finite $T'/\fm_{T'}$-algebra and is an Artinian algebra,
  so its radical $\fm_S/\fm_{T'} S$ is nilpotent, and $\fm_S^n\subset \fm_{T'} S$ for
  some $n\geq 1$.
  If $S$ is $T'$-finite, then $\fm_S^n\subset \fm_{T'} S\subset \fm_S$ for some $n$.
  Similarly, $\fm_{S'}^m \subset \fm_{T'} S' \subset \fm_{S'}$ for some $m$.
  So $\fm_{S'}^m S\subset \fm_S$ and $\fm_S^n\subset \fm_{S'} S$, and the hypothesis
  is satisfied.

  Now we prove the assertion.
  Let $M\in S\Gr\md$.
  Then
  \begin{multline*}
  u_S(M) = \ell_S(M/\fm_S M)\leq \ell_{S'}(M/\fm_S M)\leq \ell_{S'}(M/\fm_{S'}^m M)\\
  \leq \ell_{S'}(S'/\fm_{S'}^m S')\cdot \ell_{S'}(M/\fm_{S'} M).
  \end{multline*}
  That $\norm{\alpha}_S\leq \ell_{S'}(S'/\fm_{S'}^mS')\norm{\alpha}_{S'}^S$ follows easily.
  On the other hand,
  we have
  \[
  \ell_{S'}(M/\fm_{S'} M)\leq
  \ell_{S'}(M/\fm_S^n M)\leq
  \ell_{S'}(S/\fm_S^n S)\cdot u_S(M),
  \]
  and $\norm{\alpha}_{S'}^S \leq \ell_{S'}(S/\fm_S^n S)\norm{\alpha}_S$ follows
  easily.
  Hence $\norm{\alpha}_{S'}^S$ is equivalent to $\norm{\alpha}_S$.

  {\bf 3}.
  This is because
  \[
  \ell_S(M/\fm_S M) \leq \dim_k M/\fm_S M \leq \dim_k S/\fm_S\cdot \ell_S(M/\fm_S M).
  \]
\end{proof}

\begin{lemma} 
\label{cont.lem}
The following $\Bbb R$-linear maps are continuous:
\begin{enumerate}
\item[\bf 1]
$\Theta^*(S) \rightarrow \Theta^\circ(S)$;
\item [\bf 2]
$\Theta^\circ(S) \rightarrow \Theta^\wedge(\hat{S})$;
\item[\bf 3]
$f^*:\Theta(S) \rightarrow \Theta(S')$, for $f:S' \rightarrow S$, finite;
\item[\bf 4]
  $f_*:\Theta(S') \rightarrow \Theta(S)$
  given by $f_*(M)=S\otimes_{S'} M$, for $f:S' \rightarrow S$, finite;
\item[\bf 5]
$\ell_S:\Theta(S) \rightarrow \mathbb R$, when $\ell_S(S) < \infty$.
\item[\bf 6]
  $\rank_R :=\dim_{Q(R)}( Q(R) \otimes_R-):\Theta(R) \rightarrow \mathbb R$, where $R$ is
  a domain (graded or not) and $Q(R)$ is its (ungraded) field of fractions.
\end{enumerate}
\end{lemma}

\begin{proof}
We only prove {\bf 3} and leave the routine verifications of the others to the reader.

Let $\norm{\;\cdot\;}_{S'}^S$ be as in Lemma~\ref{RS.lem}.
By Lemma~\ref{RS.lem}, there exists some $r>0$ such that
$\norm{\;\cdot\;}_{S'}^S\leq r\cdot \norm{\;\cdot\;}_S$.
For $\alpha = \sum_M c_M [M]$ as a sum of indecomposable modules in $\Theta(S)$, we have
\begin{multline*}
  \norm{f^*\alpha}_{S'}
  =
  \norm{\sum_Mc_M [M]}_{S'}
  \leq
  \sum_M|c_M| \norm{M}_{S'}
  =
  \sum_M|c_M| \norm{M}^S_{S'} \\
  \leq
  r\cdot \sum_M |c_M|\norm{M}_S
  =
  r\cdot \norm{\alpha}_S,
\end{multline*}
and continuity follows.
\end{proof}

\paragraph
\label{par.plus}
Define $\Theta_+(S)$ to be the subset of $\Theta(S)$ consisting of the
$\alpha = \sum c_M [M]$ with all the $c_M \geq 0$.

\begin{lemma}
\label{lem.plus}
Suppose that $f: S' \rightarrow S$ is finite and let $\{\alpha_i \}_{i \in \mathbb N}$ be a sequence
of elements of $\Theta(S)$ such that each $\alpha _i$ is in $\Theta_+(S)$ or $-\Theta_+(S)$.
Then $\norm{ \alpha _i}_S \rightarrow 0$ if and only if $u_{S'}(f^* \alpha_i) \rightarrow 0$.
\end{lemma}

\begin{proof}
  Note that $\norm{\alpha_i}_{S}\rightarrow 0$ if and only if
  $\norm{\alpha_i}_{S'}^S\rightarrow 0$ by
  Lemma~\ref{RS.lem}.
  As $\alpha_i\in\pm \Theta_+(S)$, we have that $\norm{\alpha_i}_{S'}^S=\abs{u_{S'}(f^*(\alpha_i))}$,
  and we are done.
\end{proof}

\paragraph
For $M,N\in S\Gr \md$, we define
\begin{multline*}
\sm_N M:=\max
  \{n\in\Bbb Z_{\geq 0}\mid
  \text{$\bigoplus_{i=1}^nN[\lambda_i]$ is  a direct summand}\\
  \text{of $M$ for some $\lambda_1,\ldots,\lambda_n\in\Bbb Q$}
  \}.
\end{multline*}
For $N\in\indecomp^\circ S$, $\sm_N:S\Gr\md\rightarrow \Bbb Z$ is an additive function,
and hence induces a linear map $\sm_N:\Theta(S)\rightarrow\RR$.
More precisely, $\sm_N$ is given by $\sm_N(\sum_M c_M[M])=c_N$,
thus $\sm_N$ is continuous.

\paragraph
Let $k$ be a field of prime characteristic $p$.
Let $R=\bigoplus_{n\geq 0}R_n$
be a commutative graded $k$-algebra such that $R_0$ is an
$F$-finite Henselian local ring.
Let $\fm_R$ be the graded maximal ideal of $R$, and assume that $R/\fm_R$ is a
finite-dimensional $k$-vector space.
Let $G$ be a finite group acting on $R$ as degree-preserving $k$-algebra automorphisms
(the case that $G$ is trivial is also important in what follows).
Let $S:=R*G$.
Note that $T$ is central in $R$ and $S$.
Note also that $R/\fm_R$ and $k$ are $F$-finite, and $R$ and $S$ are finite over $T:=R^G$ \cite[(9.6)]{Hashimoto12}.
It is easy to see that $T$ is $F$-finite and Henselian.

Let $d=\dim R$, $\fd:=\log_p[k:k^p]$, and set $\delta=d+\fd$.

\paragraph
For $\alpha=\sum_{M\in \indecomp^\circ S}c_M[M] \in \Theta(S)$, define
\[
  {}^e\alpha=\sum_{M\in \indecomp^\circ S}c_M[{}^eM],
  \]
  and call it the $e$th Frobenius direct image of $\alpha$.
  We define $\NF_e(\alpha)=\frac{1}{p^{\delta e}}{}^e\alpha$.

\begin{definition}
Let
\[
\FL(\alpha):=\lim_{e\rightarrow \infty}\frac{1}{p^{\delta e}}{}^e\alpha=\lim_{e\rightarrow\infty}\NF_e(\alpha)
\]
in $\Theta(S)$, provided the limit exists.
We call $\FL(\alpha)$ the {\em Frobenius limit} of $\alpha$.
\end{definition}

\paragraph
Assume that $R$ is a domain.
As we have $\log_p [Q(R):Q(R)^p]=\delta$ by \cite[(2.3)]{Kunz},
$\rank_R {}^eM=p^{\delta e}\rank_{{}^eR}{}^eM=p^{\delta e}\rank_R M$.
It follows that $\rank_R \NF_e(\alpha)=\rank_R \alpha$ for $\alpha\in \Theta(S)$.
If $\FL(\alpha)$ exists, then $\rank_R \FL(\alpha)=\rank_R \alpha$.

\paragraph
When $I$ is a $G$-ideal in $R$, we sometimes write $\alpha /I\alpha$ for $R/I\otimes_R \alpha$.
Note that ${}^e\alpha/I({}^e\alpha)={}^e(\alpha/I^{[p^e]}\alpha)$,
where $I^{[p^e]}$ is the ideal generated by $\{a^{p^e}\mid a\in I\}$, which is a $G$-ideal.

\paragraph
\label{HK.par}
If $\fq$ is a homogeneous $\fm_T$-primary ideal of $T$,
the Hilbert--Kunz multiplicity of $M\in T\Gr\md$ \cite{Monsky} is defined by
\[
e\HK(\fq,M):
=
\lim_{e\rightarrow\infty}\frac{\ell_T(M/\fq^{[p^e]}M)}{p^{de}}
=
\lim_{e\rightarrow\infty}\frac{\ell_T(T/\fq\otimes_T {}^eM)}{p^{\delta e}}.
\]
This is an additive function, so it induces a function on $\Theta(T)$:
\[
e\HK(\fq,\alpha)
=
\lim_{e\rightarrow\infty}\frac{\ell_T(T/\fq\otimes_T {}^e\alpha)}{p^{\delta e}}
=
\lim_{e\rightarrow\infty}\ell_T(T/\fq\otimes_T \NF_e(\alpha)).
\]

By Lemma~\ref{cont.lem},
$e\HK(\fq,\alpha)=\ell_T(T/\fq\otimes_T \FL(\alpha))$,
provided $\FL(\alpha)$ exists.

Note that if $T$ is a domain then $e\HK(\fq,M)=\rank_TM \cdot e\HK(\fq,T)$.

\paragraph
Let $N\in\indecomp^\circ S$.
We define
\[
\FS_N(\alpha):=\lim_{e\rightarrow\infty}\sm_N(\NF_e(\alpha)),
\]
provided the limit exists.
We call it the generalized $F$-signature of $M$ with respect to $N$,
see \cite{HN}.
If $\FL(\alpha)$ exists, then $\FS_N(\alpha)=\sm_N(\FL(\alpha))$, since $\sm_N$ is
continuous.

\begin{example}\label{Bruns.ex}
  In \cite{Bruns}, Bruns studied the asymptotic behavior of the Frobenius
  direct images of normal affine semigroup rings;
  we follow the notation used there.
  In \cite[Theorem~3.1]{Bruns}, assume for simplicity that 
  $M$ is positive in the sense that there is a rational
  hyperplane $H$ of $\Bbb R^d$ through the origin such that
  $H\cap M=\{0\}$.
  Let $h:\Bbb R^d\rightarrow \Bbb R$ be a defining equation of $H$
  \(that is, $h^{-1}(0)=H$\) such that $h(\Bbb Z^d)\subset \Bbb Z$ and
  $h(M)\subset \Bbb Z_{\geq 0}$.
  Then $R=\bigoplus_{n\in \Bbb Z}R_n$ is positively graded
  \(that is, $R_n=0$ for $n<0$ and $R_0=K$\), where $R_n=\bigoplus_{x\in
    h^{-1}(n)\cap M}Kx$.
  Let $\fm=\bigoplus_{n>0}R_n$.
  By \cite[Theorem~3.1]{Bruns}, we immediately have that
  \[
  \FL(R)=\sum_\gamma \vol(\gamma)[\Cal C_\gamma]
  \]
  in $\Theta^\circ(R)$.
\end{example}

\section{The Frobenius limit for a group acting on a polynomial ring}\label{main.sec}

\paragraph
\label{poly.par}
Let $k$ be a field, and let $B$ be a graded polynomial ring over $k$ with the degrees
of the generators all positive integers, but \emph{not} necessarily the same.
Let $G$ be a finite group that acts \emph{faithfully} on $B$ as a graded $k$-algebra.
We can form the twisted group algebra $B*G$ and we define the Frobenius operator on it as in
(\ref{frob.para}).

Let $A=B^G$, the ring of invariants.
Let $\fm_A$ and $\fm_B$ denote the irrelevant maximal ideals of $A$ and $B$,
respectively.
Let $\hat A$ be the $\fm_A$-adic completion of $A$ and
let $\hat B$ be the $\fm_B$-adic completion of $B$ (it is also the $\fm_A$-adic completion).

Let $\V$ be the category of $\mathbb Q$-graded $kG$-modules and
let $\M$ be the category of $\mathbb Q$-graded
$B*G$-modules.

Let $\Cal F$ denote the full subcategory of $\M$ consisting of
$F\in\M$ such that $F$ is $B$-finite  and $B$-free.
In other words, $F$ is a $\mathbb Q$-graded $B*G$-lattice.

\paragraph
Let $V=\bigoplus_\lambda V_\lambda$ be an object of $\V$.
Then $V$ is a projective object of $\V$ if and only if it is so as a $kG$-module,
since $\Hom_{\V}(V,W)=\prod_\lambda\Hom_{kG}(V_\lambda,W_\lambda)$.
We denote the category of finite dimensional projective objects of $\V$ by $\P_0$.
Then clearly $\P_0=\add\{kG[\lambda]\mid \lambda\in \mathbb Q \}$,
where $[\lambda]$ denotes shift of degree by $\lambda$.

\begin{lemma}\label{c-flat.lem}
  Let $R=\bigoplus_{i\geq 0}R_i$ be a commutative positively-graded \(that is, $R_0=k$\)
  $k$-algebra.
  Let $F$ and $F'$ be graded $R$-finite $R$-free modules, and
  $h:F\rightarrow F'$ a graded $R$-homomorphism.
  Then the following are equivalent:
  \begin{enumerate}
  \item[\bf 1] $h$ is injective, and $C:=\Coker h$ is $R$-free;
  \item[\bf 2] $1\otimes h: R/\fm\otimes_R F\rightarrow R/\fm\otimes_R F'$ is
    injective;
  \end{enumerate}
  where $\fm=\bigoplus_{i>0}R_i$ is the irrelevant ideal.
\end{lemma}

\begin{proof}
  {\bf 1$\Rightarrow$2}.
  As the sequence
  \[
  0\rightarrow F\xrightarrow h F'\rightarrow C\rightarrow 0
  \]
  is exact,
  \[
  0=\Tor^R_1(R/\fm,C)\rightarrow R/\fm\otimes_R F\xrightarrow{1\otimes h}
  R/\fm\otimes_R F'
  \]
  is exact.

  {\bf 2$\Rightarrow$1}.
  Take a homogeneous free basis $f_1,\ldots,f_r$ of $F$, and
  take homogeneous elements $f_1',\ldots,f_s'$ of $F'$ such that their
  images in $C$ form a minimal set of generators for $C$.
  As
  \[
  0\rightarrow R/\fm\otimes_R F
  \rightarrow R/\fm\otimes_R F'
  \rightarrow R/\fm\otimes_R C
  \rightarrow 0
  \]
  is exact, we have that $\rank F'=r+s$, and $h(f_1),\ldots,h(f_r),f_1',\ldots,f_s'$
  generate $F'$ by the graded version of Nakayama's lemma (this applies since the grading on the modules must be discrete).
  Thus it is easy to see that this set of elements forms a free basis for $F'$.
  In particular, $h(f_1),\ldots,h(f_r)$ are linearly independent and hence
  $h$ is injective.
  Also, $C=F'/F$ is a free module with basis $f_1',\ldots,f_s'$.
\end{proof}

\begin{lemma}\label{Frobenius.lem}
  \begin{enumerate}
  \item[\bf 1]
    $P:=\{(B\otimes kG)[\lambda]\mid \lambda\in \mathbb Q \}$
    is a set of Noetherian projective objects that generate $\M$.
    In particular,
    $\P:=\add P$ is the full subcategory of Noetherian projective objects of $\M$.
  \item[\bf 2] 
    For $M\in\M$, the following are equivalent.
    \begin{enumerate}
    \item[\bf a] $M\in \Cal P$;
    \item[\bf b] $M\cong B\otimes_k V$ as graded modules, for some $V\in\Cal P_0$;
    \item[\bf c] $M\in\F$, and $M/\fm_B M\in \Cal P_0$.
    \end{enumerate}
    If these conditions are satisfied, then $M\cong B\otimes_k M/\fm_B M$ as graded modules.
    \item[\bf 3] 
      $\Cal F$ is a Frobenius category with respect to all short exact sequences \(see {\rm \cite{Happel}} for definition\),
      and $\Cal P$ is its full subcategory of projective and injective objects.
  \end{enumerate}
\end{lemma}

\begin{proof}
    {\bf 1} Obviously, each $(B\otimes_k kG)[\lambda]$ is a Noetherian object.
    On the other hand,
    \[
      \Hom_{\M}(B\otimes kG[\lambda],N)\cong \Hom_{\V}(kG[\lambda],N)\\
      \cong
      \Hom_{\Gr\Mod k}(k[\lambda],N)\cong N_{-\lambda},
    \]
    and each object of $P$ is a projective object, and $P$ generates $\M$,
    where $\Gr\Mod k$ denotes the category of graded $k$-vector spaces.

{\bf 2}.
{\bf a$\Leftrightarrow$b$\Rightarrow$c} is trivial.
We show the last assertion, assuming {\bf c}.
This also proves {\bf c$\Rightarrow$b}.
As $M/\fm_B M$ is projective in $\V$, the canonical map $M\rightarrow M/\fm_B M$ has
a splitting $j: M/\fm_B M\rightarrow M$ in $\V$.
Then, defining $\varphi : B\otimes_k M/\fm_B M\rightarrow M$ by $\varphi(b\otimes v)
=b j(v)$, $\varphi$ is
$B*G$-linear.
By Lemma~\ref{c-flat.lem}, it is easy to see that $\varphi$ is an isomorphism.

  {\bf 3}.
  By {\bf 1},
  $\P$ is the category of the projectives of $\F$, and $\F$ has
  enough projectives.
  On the other hand, $\Hom_B(?,B)$ is a dualizing functor on the exact category $\F$
  and $\P$ is mapped to itself by it.
  Thus $\P$ is also the
  category of injectives of $\F$, and $\F$ has enough injectives.
\end{proof}

\begin{lemma}\label{filtration.lem}
  Let $F\in\F$.
  Then there is a filtration
  \[
  0=F_0\subset F_1\subset\cdots\subset F_n=F
  \]
  in $\M$ such that for each $i=1,\ldots,n$,
  there exist $\lambda_i\in \mathbb Q$ and $V_i\in kG\md $
  such that $F_i/F_{i-1}\cong B\otimes_k V_i[-\lambda_i]$
  \(so $F_i$ and $F_i/F_{i-1}$ are in $\Cal F$\),
  where $kG\md $ denotes the category of finite dimensional $kG$-modules,
  and each object of $kG \md$ is viewed as an object of $\V$ of degree zero.
\end{lemma}
  
\begin{proof}
  We use induction on $\rank_B F$.
  If $\rank_B F=0$, there is nothing to prove.
  Assume that $\rank F>0$
  and take the smallest $\lambda\in \mathbb Q$ such that $F_\lambda\neq 0$.
  Set $V_1=F_\lambda[\lambda]$,
  $\lambda_1=\lambda$, and $F_1=B\otimes_k V_1[-\lambda]$.
  There is a canonical map
  \[
  q:F_1=B\otimes_k V_1[-\lambda]=B\otimes_k F_\lambda\xrightarrow{a} F,
  \]
  where $a(b\otimes f)=bf$.
  Then, by Lemma~\ref{c-flat.lem}, 
  $q$ is injective, and $C\in\Cal F$,
  where $C=\Coker q$.
  Applying the induction hypothesis to $C$, we are done.
\end{proof}

\begin{lemma}\label{kG.lem}
  Let $F\in\Cal F$ and $f\geq 0$.
  Then the following are equivalent.
  \begin{enumerate}
  \item[\bf 1] $F\cong B\otimes_k F_0$ for some $\mathbb Q$-graded $G$-module
    $F_0$ such that $F_0\cong (kG)^f$ as $G$-modules.
  \item[\bf 2] $F\cong (B\otimes_k kG)^f$ as a
    $B*G$-module.
  \item[\bf 3] $F/\fm_B F\cong (kG)^f$ as a $G$-module.
  \end{enumerate}
\end{lemma}

\begin{proof}
  {\bf 1$\Rightarrow$2$\Rightarrow$3} is trivial.
  {\bf 3$\Rightarrow$1} follows from Lemma~\ref{Frobenius.lem}, {\bf 2}.
\end{proof}

\paragraph
We denote the full subcategory of $\Cal F$ with objects the $F\in\Cal F$
satisfying the equivalent conditions in Lemma~\ref{kG.lem} by $\Cal G$.
Note that $\Cal G$ is closed under extensions and shift of degree.

\begin{lemma}\label{V-V'.lem}
  Let $V$ be a $kG$-module.
  Let $V'$ be the $k$-vector space $V$ with the trivial $G$-action.
  Then $kG\otimes V \cong kG\otimes V'$.
  Hence $kG\otimes V$ is a direct sum of copies of $kG$.
\end{lemma}

\begin{proof}
  The map $g\otimes v\mapsto g\otimes g^{-1}v$ gives a $kG$-isomorphism
  $kG\otimes V\cong kG\otimes V'$.
\end{proof}

\paragraph
From now on, we assume that $k$ is of characteristic $p$, and is $F$-finite.
We set $\fd:=\log_p[k:k^p]$ and $\delta:=d+\fd$.

\begin{lemma}\label{G-F-closed.lem}
  If $F\in\Cal G$, then ${}^eF\in\Cal G$.
\end{lemma}

\begin{proof}
  We can write $F=B\otimes_k F_0$ with $F_0\cong (kG)^f$ as a $kG$-module
  for some $f$.
  We have ${}^eF\in\F$ and
  \[
    {}^eF/\fm_B{}^eF
    \cong
        {}^e(B/\fm_B^{[p^e]}\otimes_B(B\otimes_k F_0))
        \cong
  {}^e(B/\fm_B^{[p^e]}\otimes_k F_0).
  \]
  As $F_0\cong (kG)^f$, we have that
  $B/\fm_B^{[p^e]}\otimes_k F_0\cong (kG)^{fp^{de}}$ by
  Lemma~\ref{V-V'.lem}.
  Hence
  ${}^eF/\fm_B{}^eF\cong{}^e((kG)^{fp^{de}})=(kG)^{fp^{\delta e}}$
  by Lemma~\ref{F-twist.lem}.
  By Lemma~\ref{kG.lem}, we have that ${}^eF\in\Cal G$.
\end{proof}

\begin{lemma}\label{e_0.lem}
  There exists some $e_0\geq 1$ such that for each $F\in\Cal F$ of rank $f$,
  there exists some direct summand $F'$ of ${}^{e_0}F$ in $\Cal F$ such that
  $F'\cong (B\otimes_k kG)^{fp^{\fd e_0}}$ as
  $B*G$-modules.
\end{lemma}

\begin{proof}
  Let $Q(A)$ and $Q(B)$ denote the fields of fractions of $A$ and $B$
  respectively.
  Then $Q(B)$ is a Galois extension of $Q(A)$ with Galois group $G$
  (here we use the assumption $G$ acts faithfully on $B$).
  So $u:Q(B)\otimes_{Q(A)}Q(B)'\rightarrow kG\otimes_k Q(B)'$
  given by $u(x\otimes y)=\sum_{g\in G}g^{-1}\otimes (gx)y$ is an isomorphism 
  of $(G,Q(B)')$-modules,
  where $Q(B)'$ is the field $Q(B)$ with the trivial $G$-action.
  So $Q(B)$ as a $G$-module is a direct sum of copies of $kG$.
  Thus there is at least one injective $kG$-map $kG\rightarrow Q(B)$.
  Multiplying by an appropriate element of $A\setminus \{0\}$, we get
  an injective $G$-linear map $kG\rightarrow B$.
  Its image is in $B_0\oplus B_1\oplus \cdots\oplus B_r$ for some $r\geq 1$,
  and it is a direct summand, since $kG$ is an injective module.
  Then by the Krull--Schmidt theorem, there is a graded $kG$-direct summand $E_0$ of
  $B$ which is isomorphic to $kG$ as a $G$-module.
  The argument so far,
  which we have given for the convenience of the reader, can be found in \cite{Symonds}.

  We can take $e_0$ sufficiently large that $E_0\cap \fm_B^{[p^{e_0}]}=0$
  for degree reasons, so $E_0\rightarrow B/\fm_B^{[p^{e_0}]}$ is injective.
  We claim that this choice of $e_0$ has the required property.

  Let $V$ be any finite-dimensional $kG$-module.
  Then the inclusion $E_0\hookrightarrow B$ induces a split monomorphism
  $\phi: {}^{e_0}(E_0\otimes_k V)\rightarrow {}^{e_0}(B\otimes_k V)$.
  Note that the composite
  \[
    {}^{e_0}(E_0\otimes_k V)\xrightarrow{\phi}
    {}^{e_0}(B\otimes_k V)
    \rightarrow
    B/\fm_B\otimes_B {}^{e_0}(B\otimes_k V)
    \cong
        {}^{e_0}(B/\fm_B^{[p^{e_0}]}\otimes_k V)
  \]
  is injective, since ${}^{e_0}(?\otimes_k V)$ is an exact functor.
  Note that ${}^{e_0}(E_0\otimes_k V)\cong (kG)^{p^{\fd e_0}\dim_k V}$ as $G$-modules.
  By Lemma~\ref{c-flat.lem}, it is easy to see that
  \[
  B\otimes_k {}^{e_0}(E_0\otimes_k V)\rightarrow {}^{e_0}(B\otimes_k V)
  \]
  given by $b\otimes m\mapsto b\phi(m)$ is an injective map of $\Cal F$
  whose cokernel $D_V$ lies in $\Cal F$.
  As $B\otimes_k {}^{e_0}(E_0\otimes_k V)\in\Cal G\subset \P$, we have a
  decomposition
  \[
    {}^{e_0}(B\otimes_k V[\lambda])=B\otimes_k {}^{e_0}(E_0\otimes_k V)[\lambda/p^{e_0}]
    \oplus D_V[\lambda/p^{e_0}].
  \]
  So if $F\cong B\otimes_k V[\lambda]$ for some finite-dimensional $kG$-module $V$ and
  $\lambda\in\Bbb Q$, the lemma holds.

    Now let
    \[
    0\rightarrow E\rightarrow F\rightarrow H\rightarrow 0
    \]
    be a short exact sequence in $\Cal F$ such that the assertion of the
    lemma (for our $e_0$) is satisfied for $E$ and $H$.
    That is, ${}^{e_0}E$ has a direct summand $E'$ such that $E'\cong
    (B\otimes_k kG)^{\oplus p^{\fd e_0}\rank E}$ as a
    $B*G$-module,
    and
    ${}^{e_0}H$ has a direct summand $H'$ such that $H'\cong
    (B\otimes_k kG)^{\oplus p^{\fd e_0}\rank H}$ as a $(G,B)$-module.
    As $H'$ is a projective object of $\Cal F$, the inclusion $H'\hookrightarrow
    H$ lifts to $H'\hookrightarrow F$.
    So we have a commutative diagram of
    $B*G$-modules, with exact rows and columns
    \[
    \xymatrix{
      & 0 \ar[d] & 0 \ar[d] & 0 \ar[d] & \\
      0 \ar[r] & E' \ar[r] \ar[d] & E'\oplus H'\ar[r] \ar[d] &
      H' \ar[d] \ar[r] & 0 \\
      0 \ar[r] & E \ar[r] \ar[d] & F \ar[r]\ar[d] & H \ar[r]\ar[d] & 0 \\
      0 \ar[r] & E'' \ar[r] \ar[d] & F'' \ar[r]\ar[d] & H'' \ar[r]\ar[d] & 0 \\
      & 0 & 0 & 0
    }.
    \]
    As $E'$ and $H'$ are direct summands of $E$ and $H$, respectively,
    we have that $E''\in\Cal F$ and $H''\in\Cal F$.
    So $F''\in\Cal F$, and hence $E'\oplus H'$ is a direct summand of
    $F$ by Lemma~\ref{Frobenius.lem}.
    As $E'\oplus H'\cong (B\otimes_k kG)^{\oplus(p^{\fd e_0}(\rank_B E+\rank_B H))}$
    and $\rank_B E+\rank_B H=\rank_B F$, we conclude that
    the assertion of the lemma is also true for $F$.

    Now by Lemma~\ref{filtration.lem}, we are done.
\end{proof}

\begin{proposition}\label{crutial.prop}
  There exists some $c>0$ and $0\leq \alpha<1$ such that for any
  $F\in\Cal F$ of rank $f$ and any $e\geq 0$, there exists some decomposition
  \begin{equation}\label{crutial.eq}
    {}^e F\cong F_{0,e}\oplus F_{1,e}
  \end{equation}
    such that $F_{1,e}\in\Cal G$ and $\rank_B F_{0,e}\leq c\alpha^e fp^{\delta e}$.
\end{proposition}

\begin{proof}
  If the dimension $d=0$, then $A=B=k$ and $G$ is trivial, and this case is obvious,
  since we may set $c=1$, $\alpha=0$, $F_{0,e}=0$ and $F_{1,e}={}^eF$ for each $e$.

  So we may assume that $d\geq 1$.
  Take $e_0$ as in Lemma~\ref{e_0.lem}, and set $\alpha:=(1-\lvert G\rvert\cdot 
  p^{-de_0})^{1/e_0}$
  so that $0 \leq \alpha<1$.
  Set $c=\alpha^{-e_0}>0$.

  We prove the existence of a decomposition by induction on $e\geq 0$.

  If $0\leq e<e_0$, then we set $F_{0,e}={}^eF$ and $F_{1,e}=0$.
  As we have $\rank_B F_{0,e}=fp^{\delta e}$ and
  $c\alpha^e=\alpha^{e-e_0}>1$, we are done.

  Now assume that $e\geq e_0$.
  By the induction hypothesis, we have a decomposition
  \[
    {}^{e-e_0}F\cong F_{0,e-e_0}\oplus F_{1,e-e_0}
    \]
    such that $F_{1,e-e_0}\in\Cal G$ and
    $\rank_B F_{0,e-e_0}\leq c\alpha^{e-e_0}fp^{\delta (e-e_0)}$.
    Then 
    \[
      {}^eF\cong {}^{e_0}F_{0,e-e_0}\oplus {}^{e_0}F_{1,e-e_0}.
      \]
      By Lemma~\ref{G-F-closed.lem}, that ${}^{e_0}F_{1,e-e_0}\in\Cal G$.
      Moreover, 
      \[
      \rank_B{}^{e_0}F_{0,e-e_0}
      =
      p^{\delta e_0}\rank_B F_{0,e-e_0}.
      \]
      By the choice of $e_0$, there is a decomposition
      \[
        {}^{e_0}F_{0,e-e_0}\cong F'\oplus F''
        \]
        such that $F'\in\Cal G$ and
        $\rank_B F'=\lvert G \rvert \cdot p^{\fd e_0}\rank_BF_{0,e-e_0}$.

        Now let $F_{0,e}:=F''$ and $F_{1,e}:={}^{e_0}F_{1,e-e_0}\oplus F'$.
        As ${}^{e_0}F_{1,e-e_0}\in\Cal G$ and $F'\in\Cal G$, we have
        $F_{1,e}\in\Cal G$.
        On the other hand,
\begin{multline*}
        \rank_B F_{0,e}=\rank_B {}^{e_0}F_{0,e-e_0}-\rank_B F'
        =
        (p^{\delta e_0}-|G|\cdot p^{\fd e_0})\rank_BF_{0,e-e_0}\\
        \leq
        \alpha^{e_0}p^{\delta e_0}c\alpha^{e-e_0}fp^{\delta (e-e_0)}
        =c\alpha^efp^{\delta e},
\end{multline*}
        and we are done.
\end{proof}

\begin{theorem}
\label{thm.FLBG}
For any $B*G$-module $F$ that is free of rank $f$ over $B$ we have
\[
\FL(F)=\frac{f}{|G|} [B* G]
\]
in $\Theta^\circ (B * G)$ and the analogous formula 
\[
\FL(\hat{F})=\frac{f}{|G|} [\hat{B}* G]
\]
in $\Theta^\wedge (\hat{B} * G)$.
\end{theorem}

\begin{proof}
From Proposition~\ref{crutial.prop}, we have 
\begin{equation*} 
\frac{[{}^eF]}{p^{\delta e}} - \frac{f}{|G|} [B * G] = \left( \frac{[F_{1,e}]}{p^{\delta e}} - \frac{f}{|G|}[B* G] \right)+ \frac{[F_{0,e}]}{p^{\delta e}}.  \tag{$\dagger$}\label{eF}
\end{equation*}
Notice that $[F_{0,e}]/p^{\delta e} \in \Theta^\circ_+(B * G)$ and $\lim_{e \to \infty} \rank_B([F_{0,e}]/p^{\delta e}) =0$. But $F_{0,e}$ is free as a $B$-module, so $u_B(F_{0,e})=\rank_B(F_{0,e})$. It follows from Lemma~\ref{lem.plus} that $\lim_{e \to \infty} \norm{[F_{0,e}]/p^{\delta e}} _{B * G} =0$. 

By Lemma~\ref{kG.lem},
the term $[F_{1,e}]/p^{\delta e}$ is of the form $a_e[B * G]$ for some
number $a_e$; taking ranks shows that
$
\lim_{e \to \infty} a_e = f/|G|.
$
Thus
\[
\lim_{e \to \infty} (\frac{[F_{1,e}]}{p^{\delta e}}-\frac{f}{|G|}[B * G]) =0
\]
and the first part of the theorem is proved.

The second part follows from Lemma~\ref{cont.lem}.
\end{proof}

\begin{lemma}\label{easy.lem}
  $B\cong (B\otimes_k kG)^G$ as graded $A$-modules.
  More explicitly, $b\mapsto \sum_g gb\otimes g$ gives a graded $A$-isomorphism.
  The inverse is given by $\sum_g b_g\otimes g\mapsto b_e$.
\end{lemma}

\begin{proof}
Easy.
\end{proof}

\begin{lemma}
\label{fix.lem}
  For any $B * G$-module $M$, $\rank_AM^G=\rank_B M$.
\end{lemma}

\begin{proof}
  It is well known that $Q(B) * G$ is isomorphic to a matrix ring over $Q(A)$
  (\cite[28.3]{CR}), hence $Q(B)$ is its only indecomposable module.
  Thus $$Q(A) \otimes_A M^G \cong (Q(A) \otimes_A M)^G \cong (Q(B) \otimes_B M)^G
  \cong (Q(B)^m)^G \cong Q(A)^m,$$
where $m = \rank_B M$.
\end{proof}

\begin{theorem}
\label{thm.FLA}
For any
$B*G$-module
$F$ that is free of rank $f$ over $B$ we have
\[
\FL(F^G)=\frac{f}{|G|} [B]
\]
in $\Theta^\circ (A)$ and
\[
\FL(\hat F^G)=\frac{f}{|G|} [\hat B]
\]
$\Theta^\wedge (\hat{A})$, where $A=B^G$.
\end{theorem}

\begin{proof}
From the proof of Theorem~\ref{thm.FLBG} we have $[{}^eF]/p^{\delta e} = a_e[B \otimes_k kG] + [F_{0,e}]/p^{\delta e}$, where $\lim_{e \to \infty} a_e =f/|G|$.
Applying the fixed point functor and using Lemma~\ref{easy.lem} yields 
\[
[{}^eF^G]/p^{\delta e}=a_e[B]+[F^G_{0,e}]/p^{\delta e}.
\]
The theorem will follow once we can show that $\lim_{e\rightarrow \infty}u_A([F^G_{0,e}]/p^{\delta e})=0$, since this takes place in $\Theta_+(A)$.

Applying $u_A$ gives
\[
u_A([{}^eF^G]/p^{\delta e})= u_A(a_e[B])+u_A([F^G_{0,e}]/p^{\delta e}).
\]
Clearly,
\[
\lim_{e\rightarrow \infty}u_A(a_e[B])= (f/|G|)u_A(B)=(f/|G|)\dim_k B/\mathfrak m_A B.
\]
Now we use the Hilbert-Kunz multiplicity (see (\ref{HK.par})).
\[
\lim_{e\rightarrow \infty}u_A \left( \frac{[{}^eF^G]}{p^{\delta e}} \right)=e\HK(\mathfrak m _A,F^G)=\rank_A(F^G) \cdot e\HK(\mathfrak m_A,A).
\]
But $\rank_A(F^G)=\rank_B(F)
=f
$, by Lemma~\ref{fix.lem}.

It was shown by Watanabe and Yoshida \cite[2.7]{WY} that $e\HK(\mathfrak m_A,A)=\frac{1}{|G|} \ell_B(B/\mathfrak m_AB)$, and this right hand side is equal to
$\frac{1}{|G|}\dim_k B/\mathfrak m_A B$.
Combining these, we see that $\lim_{e\rightarrow \infty}u_A([F^G_{0,e}]/p^{\delta e})=0$, as required. 
\end{proof}

\begin{remark}
When $p$ does not divide $|G|$ it is easy to see that the map induced by the fixed point functor $\Theta^\circ(B * G) \rightarrow \Theta^\circ (A)$ is continuous, so Theorem~\ref{thm.FLA} follows immediately from Theorem~\ref{thm.FLBG}.
\end{remark}

\section{Applications}
\label{applications.sec}

We continue to use the notation of (\ref{poly.par}).

\begin{theorem}\label{main.thm}
  Let $k$ be a field of characteristic $p>0$
  such that $[k:k^p]<\infty$,
  and let $V$ be a faithful $G$-module.
  Let $k=V_0,V_1,\ldots,V_n$ be the simple $kG$-modules.
  For each $i$, let
  $P_i\rightarrow V_i$ be the projective cover, and set
  $M_i:=(B\otimes_k P_i)^G$.
  Let $F$ be a $\Bbb Q$-graded 
  $B$-finite $B$-free $B*G$-module.
  Then the $F$-limit of $[F^G]$ exists in $\Theta^\circ(A)$, where $A=B^G$, and
  \[
  \FL([F^G])=\frac{f}{\abs{G}}[ B]
  =\frac{f}{\abs{G}}\sum_{i=0}^n \frac{\dim_k V_i}{\dim_k \End_{kG}(V_i)}[ M_i],
  \]
  where $f=\rank_B F$. An analogous formula holds for $\FL([\hat F^G])$ in $\Theta^\wedge(\hat{A})$.
\end{theorem}

\begin{proof}
 The first equality is just Theorem~\ref{thm.FLA}. 

  We can write $kG=\bigoplus_{i=0}^n P_i^{\oplus u_i}$ for some $u_i\geq 0$, so $B\cong (B\otimes_k kG)^G\cong \bigoplus_{i=0}^n M_i^{\oplus u_i}$. Applying $\dim_k \Hom_{kG}(-,V_i)$  to the first equality shows that $u_i=\dim_k(V_i)/\dim_k\End_{kG}(V_i)$.
 \end{proof} 

\begin{corollary}\label{main.cor}
  Under the conditions of {\rm Theorem~\ref{main.thm}}, we have
  \[
  \FL([ A])=\frac{1}{\abs{G}}[ B]
  =\frac{1}{\abs{G}}
  \sum_{i=0}^n \frac{\dim_k V_i}{\dim_k \End_{kG}(V_i)}[ M_i]
  \]
  in $\Theta^\circ(A)$
  and similarly after completion.
  \qed
\end{corollary}

\paragraph\label{no-pseudo-reflection.par}
Let the notation be as in Theorem~\ref{main.thm}.
We say that the action of $G$ on $B$ (or on $X:=\Spec B$) is {\em small} if
there is a $G$-stable open subset $U$ of $X$ such that
the action of $G$ on $U$ is free, and the codimension of $X\setminus U$
in $X$ is at least two.

For $g\in G$, let $X_g$ be the locus in $X$ that the action of $g$ and
the identity map agree.
Note that $X_g$ is a closed subscheme of $X$.
If all the generators of $B$ are in degree one, then $X_g$ is nothing but
the eigenspace in $V$ with eigenvalue $1$ of the action of $g$ on $V$, where
$V=B_1$.
We say that $g$ is a pseudo-reflection if the codimension of $X_g$ in $X$ is one.
The action of $G$ on $B$ is small if and only if $G$ does not have a pseudo-reflection.

Now assume further that the action of $G$ on $B$ is small.

\begin{theorem}\label{equivalence.thm}
  Let the notation be as in {\rm(\ref{no-pseudo-reflection.par})}.
  Then $(B\otimes_{A}?):\Ref(A)\rightarrow \Ref(G,B)$ is an
  equivalence with quasi-inverse $(?)^G:\Ref(G,B)\rightarrow \Ref(A)$,
  where $\Ref(A)$ denotes the category of reflexive $A$-modules,
  and $\Ref(G,B)$ denotes the full subcategory of $(G,B)\md$ consisting of
  $(G,B)$-modules which are reflexive as $B$-modules.
  A similar assertion for $\hat A\rightarrow \hat B$ also holds.
\end{theorem}

\begin{proof}
This is a special case of \cite[(14.24)]{Hashimoto12}.
See also \cite[(2.4)]{HN}.
\end{proof}

Using Theorem~\ref{equivalence.thm}, we can obtain the following equivalences.

\begin{corollary}\label{indec-isom.cor}
  Let the notation be as in {\rm(\ref{no-pseudo-reflection.par})}.
  For $V\in kG\md$, define $M_V:=(B\otimes_k V)^G$.
  \begin{enumerate}
  \item[\bf 1] For $V\in G\md$, the following are equivalent.
    \begin{enumerate}
    \item[\bf a] $V$ is an indecomposable $kG$-module.
    \item[\bf b] $B\otimes_k V$ is an indecomposable object in $(B*G)\md$.
    \item[\bf \^b] $\hat
      B\otimes_k V$ is an indecomposable object in $(\hat B*G)\md$.
    \item[\bf c] $M_V$ is an indecomposable $A$-module.
    \item[{\bf \^c}] $\hat M_V$ is an indecomposable $\hat A$-module.
    \end{enumerate}
  \item[\bf 2] Let $V,V'\in G\md$.
    Then the following are equivalent.
    \begin{enumerate}
    \item[\bf a] $V\cong V'$ in $G\md$.
    \item[\bf b] $B\otimes_k V\cong B\otimes_k V'$ in $(B*G)\md$.
      \item[{\bf \^b}]
        $\hat B\otimes_k V\cong \hat B\otimes_k V'$ in $(\hat B*G)\md$.
      \item[\bf c] $M_V\cong M_{V'}$ as $A$-modules.
      \item[{\bf \^c}] $\hat M_V\cong \hat M_{V'}$.
    \end{enumerate}
  \end{enumerate}
\end{corollary}

\begin{proof}
  We only prove {\bf 1}.

  {\bf b$\Rightarrow $a}.
  This is because $B\otimes_k ?$ is a faithful exact functor from $G\md$
  to $B*G\md$.

  {\bf a$\Rightarrow$b}.
  This is because $B/\fm_B\otimes_B ?$ is an additive functor from the
  category of $B$-finite $B$-free $B*G$-modules to $kG \md$, which sends a nonzero
  object to a nonzero object.

  {\bf a$\Leftrightarrow${\bf \^b}} is similar.
  {\bf b$\Leftrightarrow$\bf c} and
  {{\bf \^b}$\Leftrightarrow${\bf \^c}} are by Theorem~\ref{equivalence.thm}.
\end{proof}

\begin{theorem}
  Let the notation be as in {\rm(\ref{no-pseudo-reflection.par})},
  so in particular the action of $G$ on $B$ is small.
  Then for each $0\leq i,j\leq n$,
  $\FS_{ M_j}( M_i)$ exists, and
  \[
  \FS_{ M_j}( M_i)=\frac{(\dim_k P_i)(\dim_k V_j)}{\abs{G} \dim_k \End_{kG}(V_i)}.
  \]
  A similar formula holds in the complete case.
\end{theorem}
    
\begin{proof}
  By Theorem~\ref{main.thm}, $\FS_{ M_j}( M_i)$ exists and
  \[
  \FS_{ M_j}( M_i)=\sm_{ M_j}(\FL( M_i))
  =\frac{\rank_{B}(B\otimes_k P_i)}{\abs{G}}\sum_{l=0}^n\frac{\dim_k V_l}{\dim_k \End_{kG}(V_i)}\sm_{ M_j}[ M_l].
  \]
  Because each $P_l$ is indecomposable and $P_l\cong P_j$ if and only if $l=j$, it follows from Corollary~\ref{indec-isom.cor} that
  each $M_l$ is indecomposable and $M_j\cong  M_l$
  (after shift of degree)
  if and only if $l=j$.
  This shows that $\sm_{ M_j}[ M_l]=\delta_{jl}$ (Kronecker's delta).
  The theorem follows.
\end{proof}

\begin{corollary}[{\cite[(3.9)]{HN}}]\label{HN-generalization.cor}
  Let the notation be as in {\rm(\ref{no-pseudo-reflection.par})} and assume that
  $k$ is algebraically closed and that $\abs{G}$ is not divisible by the characteristic of $k$.
  Then, for each $0\leq i,j\leq n$,
  $\FS_{\hat M_j}(\hat M_i)$ exists, and
  \[
  \FS_{\hat M_j}(\hat M_i)=\frac{(\dim_k V_i)(\dim_k V_j)}{\abs{G}}.
  \]
\end{corollary}

\begin{proof}
  This is because $P_i\cong V_i$, by Maschke's theorem.
\end{proof}

\begin{corollary}[{\cite[Corollary~2]{Broer}}, {\cite[Corollary~3.3]{Yasuda}}]
  Let the notation be as in {\rm(\ref{no-pseudo-reflection.par})}.
  If $p$ divides $\abs G$, then none of $\hat A$, $A_{\fm_A}$, nor
  $A$ is weakly $F$-regular.
\end{corollary}

\begin{proof}
  By Corollary~\ref{indec-isom.cor}, {\bf 1}, $\hat M_j$ is
  indecomposable
  for $j=0,1,\ldots,n$.
  By Corollary~\ref{indec-isom.cor}, {\bf 2}, $\hat M_j=\hat M_{P_j}\cong \hat M_{k}
  =\hat A$ if and only if $P_j\cong k$.
  This happens if and only if $j=0$ and $P_0\rightarrow k$ is an isomorphism.
  This is equivalent to saying that $p$ does not divide $\abs{G}$ and $j=0$.
  By our assumption, $\sm_{\hat A}(\hat M_j)=0$ for $j=0,\ldots,n$.
  So by Theorem~\ref{main.thm},
  \[
  \FS_{\hat A}(\hat A)=\sm_{\hat A}(\FL(\hat A))=\sum_{j=0}^n\frac{\dim_k V_j}{\dim_k \End_{kG}(V_i)}\sm_{\hat A}(\hat
  M_j)=0.
  \]
  Since $\FS_{\hat A}(\hat A)$ is just the $F$-signature of $\hat A$ of
  Huneke--Leuschke \cite{HL},
  we see that $\hat A$ is not strongly $F$-regular, by the theorem of Aberbach and Leuschke \cite{AL}.
  So $\hat A$ cannot be a direct summand subring of the regular local ring
  $\hat B$.
  As a weakly $F$-regular ring is a splinter \cite[(5.17)]{HH},
  $\hat A$ is not weakly $F$-regular.
  By smooth base change \cite[(7.3)]{HH2}, $A_{\fm_A}$ is not weakly $F$-regular.
  It follows that $A$ is not weakly $F$-regular.
\end{proof}

M{\sc itsuyasu} H{\sc ashimoto},  
  Okayama University, Okayama 700--8530, JAPAN
  
  {\small \tt mh@okayama-u.ac.jp}

  \medskip
  
  P{\sc eter} S{\sc ymonds},
  University of Manchester, 
  United Kingdom

  {\small \tt Peter.Symonds@manchester.ac.uk}

\end{document}